\documentclass[a4paper,11pt]{article}
\usepackage{amsmath, amsfonts, amscd, amssymb, amsthm, enumerate, xypic}

\def\Diag{\text{Diag}}
\def\Mat{\text{M}}
\def\GL{\text{GL}}
\def\SL{\text{SL}}
\def\bfB{\mathbf{B}}

\def\card{\#\,}
\newcommand{\Ker}{\operatorname{Ker}}
\newcommand{\Com}{\operatorname{Com}}
\newcommand{\Vect}{\operatorname{span}}
\newcommand{\im}{\operatorname{Im}}
\newcommand{\ad}{\operatorname{ad}}
\newcommand{\tr}{\operatorname{tr}}
\newcommand{\rk}{\operatorname{rk}}
\newcommand{\codim}{\operatorname{codim}}
\renewcommand{\setminus}{\smallsetminus}

\def\K{\mathbb{K}}

\def\N{\mathbb{N}}

\renewcommand{\L}{\mathbb{L}}

\def\calA{\mathcal{A}}
\def\calB{\mathcal{B}}

\def\calG{\mathcal{G}}
\def\calH{\mathcal{H}}

\def\calL{\text{End}}

\def\calP{\mathcal{P}}

\def\calV{\mathcal{V}}
\def\calW{\mathcal{W}}

\def\lcro{\mathopen{[\![}}
\def\rcro{\mathclose{]\!]}}

\theoremstyle{definition}
\newtheorem{Def}{Definition}
\newtheorem{Not}[Def]{Notation}

\theoremstyle{plain}
\newtheorem{theo}{Theorem}
\newtheorem{prop}[theo]{Proposition}

\newtheorem{lemme}[theo]{Lemma}

\theoremstyle{plain}

\theoremstyle{remark}
\newtheorem{Rems}{Remarks}
\newtheorem{Rem}[Rems]{Remark}

\title{To what extent is a large space of matrices not closed under the product?}
\author{Cl\'ement de Seguins Pazzis
\footnote{Professor of Mathematics at Lyc\'ee Priv\'e Sainte-Genevi\`eve, 2, rue
de l'\'Ecole des Postes, 78029 Versailles Cedex, FRANCE.}
\footnote{e-mail address: dsp.prof@gmail.com}}

\begin{document}

\thispagestyle{plain}
\maketitle
\begin{abstract}
Let $\K$ denote a field. Given an arbitrary linear subspace $V$ of $\Mat_n(\K)$ of codimension lesser than $n-1$,
a classical result states that $V$ generates the $\K$-algebra $\Mat_n(\K)$.
Here, we strengthen this statement in three ways:
we show that $\Mat_n(\K)$ is spanned by the products of the form $AB$ with $(A,B)\in V^2$;
we prove that every matrix in $\Mat_n(\K)$ can be decomposed into a product of matrices of $V$;
finally, when $V$ is a linear hyperplane of $\Mat_n(\K)$ and $n>2$, we show that every matrix in $\Mat_n(\K)$ is a product of two elements of $V$.
\end{abstract}



\section{Introduction}

In this paper, $\K$ denotes an arbitrary field, $n$ a positive integer and
$\Mat_n(\K)$ the algebra of square matrices of order $n$ with coefficients in $\K$.
For $(p,q)\in \N^2$, we also denote by $\Mat_{p,q}(\K)$ the vector space of matrices with $p$ rows, $q$
columns and entries in $\K$. For $(i,j)\in \lcro 1,n\rcro \times \lcro 1,p\rcro$, we let $E_{i,j}$ denote the elementary matrix
of $\Mat_{n,p}(\K)$ with entry $1$ at the $(i,j)$ spot and zero elsewhere.
We set $\frak{sl}_n(\K):=\bigl\{M \in \Mat_n(\K) : \; \tr M=0\bigr\}$.
The standard lie bracket on $\Mat_n(\K)$ will be written $[-,-]$.
We equip $\Mat_n(\K)$ with the non-degenerate symmetric bilinear map $b : (A,B) \mapsto \tr(AB)$.
Given a subset $\calA$ of $\Mat_n(\K)$, its orthogonal subspace for $b$ will be written $\calA^\bot$.

\vskip 2mm
\noindent Given a vector space $E$ over $\K$, we let $\calL(E)$ denote the ring of linear endomorphisms on $E$,
and, if $E$ is finite-dimensional, we also write $\frak{sl}(E):=\bigl\{u \in \calL(E) : \; \tr(u)=0\bigr\}$.

\vskip 3mm
Here, we will deal with linear subspaces of $\Mat_n(\K)$ with a small \emph{codimension}
in $\Mat_n(\K)$ and some properties they share related to the product of matrices.
Our starting point is a result that is well-known to specialists of representations of algebras:
a strict subalgebra of $\Mat_n(\K)$ must have a codimension greater than or equal to $n-1$.
Here is a proof using a theorem of Burnside:

\begin{proof}
Let $\calA$ be a strict subalgebra of $\Mat_n(\K)$. Choose an algebraic closure $\mathbb{L}$ of $\K$.
Then $\calA_{\mathbb{L}}:=\calA \otimes_\K \mathbb{L}$ is an $\mathbb{L}$-subalgebra
of $\Mat_n(\K) \otimes_\K \mathbb{L}$.
Moreover $\dim_{\mathbb{L}} \calA_{\mathbb{L}}=\dim_\K \calA_\K$. Hence
$\calA_{\mathbb{L}}$ is a strict subalgebra of $\Mat_n(\K) \otimes_\K \mathbb{L} \simeq \Mat_n(\mathbb{L})$, hence Burnside's theorem (see \cite{RadjRosen} Theorem 1.2.2 p.4) shows that $\L^n$ is not a simple $\calA_{\mathbb{L}}$-module.
It follows that we may find a linear embedding of $\calA_{\mathbb{L}}$ into the space of matrices of the form
$$\begin{bmatrix}
A & B \\
0 & C
\end{bmatrix} \quad \text{with $A \in \Mat_p(\mathbb{L})$, $B \in \Mat_{p,n-p}(\mathbb{L})$ and $C \in \Mat_{n-p}(\mathbb{L})$,}$$
hence $\codim_{\Mat_n(\mathbb{L})} \calA_{\mathbb{L}} \geq p(n-p) \geq n-1$.
\end{proof}

As a consequence, if a linear subspace $V$ of $\Mat_n(\K)$ has codimension lesser than $n-1$,
then it is not closed under the matrix product, and, better still, $V$ generates the $\K$-algebra $\Mat_n(\K)$.
In the present paper, we aim at strengthening this result in various ways.

\begin{Not}
Given a subset $V$ of $\Mat_n(\K)$, we set
$$V^{(2)}:=\bigl\{AB\mid (A,B)\in V^2\bigr\} \quad \text{and} \quad
V^{(\infty)}:=\bigl\{A_1A_2\cdots A_p\mid p \in \N, \; (A_1,\dots,A_p)\in V^p\bigr\}$$
i.e. $V^{(\infty)}$ is the \emph{sub-semigroup} of $\bigl(\Mat_n(\K),\times\bigr)$ generated by $V$.
\end{Not}

\begin{theo}\label{LC2}
Let $V$ be a linear subspace of $\Mat_n(\K)$ such that $\codim V<n-1$. \\
Then every matrix of $\Mat_n(\K)$ is a sum of matrices in $V^{(2)}$.
\end{theo}

\noindent
Notice that $$W_1:=\biggl\{\begin{bmatrix}
\alpha & M \\
0 & L
\end{bmatrix} \mid M \in \Mat_{n-1}(\K), \, L \in \Mat_{1,n-1}(\K), \alpha \in \K\biggr\}$$
is a subalgebra of codimension $n-1$ hence the upper bound in Theorem \ref{LC2} is tight.

\begin{theo}\label{prodall}
Let $V$ be a linear subspace of $\Mat_n(\K)$ such that $\codim V<n-1$. \\
Then $V$ generates the semigroup $\bigl(\Mat_n(\K),\times\bigr)$, i.e. $\Mat_n(\K)=V^{(\infty)}$.
\end{theo}

\noindent Again, the case of $W_1$ above shows that the upper bound $n-1$ is tight.

\begin{theo}\label{prod2}
Assume $n \geq 3$ and let $V$ be a (linear) hyperplane of $\Mat_n(\K)$. \\
Then $\Mat_n(\K)=V^{(2)}$.
\end{theo}

\noindent So far, we have not found any linear subspace $V$ of $\Mat_n(\K)$
such that $\codim V<n-1$ and $V^{(2)}\neq \Mat_n(\K)$.

\vskip 3mm
\noindent Theorems \ref{LC2} and \ref{prodall} will be respectively proven in Sections
\ref{genproductpairs} and \ref{productofanysize}, whilst
Section \ref{hyperplanes} is devoted to the proof of Theorem \ref{prod2}: there, we will
also solve the special case $n=2$ (i.e.\ we will determine, up to conjugation, all
the hyperplanes $H$ of $\Mat_2(\K)$ for which $H^{(2)}=\Mat_2(\K)$).
Those three sections are essentially independent one from the others.

\section{The linear subspace spanned by products of pairs}\label{genproductpairs}

\subsection{Products of pairs from the same subspace}

Our proof of Theorem \ref{LC2} is based on the following result:

\begin{prop}\label{bracket}
Let $V$ be a linear subspace of $\Mat_n(\K)$ such that $\codim V<n-1$.
Then
$$\frak{sl}_n(\K)=\Vect \bigl\{[A,B] \mid (A,B)\in V^2\bigr\}.$$
\end{prop}

\begin{proof}
Set $F:=\Vect \bigl\{[A,B] \mid (A,B)\in V^2\bigr\}$.
The inclusion $F \subset \frak{sl}_n(\K)$ is trivial.
Conversely, let $A \in F^\bot$ and $B \in V$.
Then, for every $C \in V$, one has
$\tr(A[B,C])=0$ hence $\tr([A,B]C)=0$.
This shows $\ad_A : M \mapsto [A,M]$ maps $V$ into $V^\bot$.
By the rank theorem, we deduce that
$$\dim \Ker \ad_A+\dim V^\bot \geq \dim V$$
hence
$$2\,\codim V \geq \codim \Ker \ad_A.$$
Assume that $A$ is not a scalar multiple of the unit matrix $I_n$.
Denote by $P_1,\dots,P_p$ its elementary factors, with $P_p \mid P_{p-1} \mid \dots \mid P_1$,
and $d_i:=\deg P_i$.
Then the Frobenius theorem on the dimension of the centralizer of a matrix
(Theorem 19 p.111 of \cite{Jacobson}) shows that
$$\dim \Ker \ad_A=\sum_{k=1}^p (2k-1)\,d_k=\sum_{1 \leq i,j \leq p} d_{\max(i,j).}$$
Therefore
$$2\codim V \geq \codim \Ker \ad_A=\sum_{1 \leq i,j \leq p} \bigl(d_id_j-d_{\max(i,j)}\bigr) \geq
d_1^2-d_1+2\sum_{i=2}^p d_i(d_1-1).$$
However $d_1 \geq 2$ since $A$ is not a scalar multiple of $I_n$,
hence
$$2\codim V \geq \codim \Ker \ad_A \geq 2d_1-2+2\sum_{i=2}^p d_i=2n-2.$$
This contradicts the initial assumption on $V$. Hence
$F^\bot \subset \Vect(I_n)$ and therefore $\frak{sl}_n(\K)=\Vect(I_n)^\bot \subset F$.
\end{proof}

From there, proving Theorem \ref{LC2} is easy.
Let $V$ be a linear subspace of $\Mat_n(\K)$ such that $\codim V < n-1$.
Then Proposition \ref{bracket} shows that $\frak{sl}_n(\K) \subset \Vect V^{(2)}$.
However, if $\frak{sl}_n(\K)=\Vect V^{(2)}$, then we would have
$\forall (A,B)\in V^2, \; \tr(AB)=0$, hence $V \subset V^\bot$ which would imply that
$\codim V \geq \frac{n^2}{2}$, in contradiction with the hypothesis $\codim V < n-1$.
Since $\frak{sl}_n(\K)$ is a hyperplane of $\Mat_n(\K)$, this proves
$\Vect V^{(2)}=\Mat_n(\K)$.

\subsection{Products of pairs from two different subspaces}

In this short section, we will diverge slightly from the main theme of this paper.
Our aim is the following result, which looks analogous to Theorem \ref{LC2} but
neither generalizes it nor follows from it.

\begin{prop}
Let $V$ and $W$ be two linear subspaces of $\Mat_n(\K)$.
\begin{enumerate}[(a)]
\item If $\codim V+\codim W<n$, then $\Mat_n(\K)$
is spanned by $V \cdot W:=\bigl\{BC\mid (B,C)\in V \times W\bigr\}$.
\item If $\codim V+\codim W=n$ and $\Mat_n(\K)$ is not spanned by $V \cdot W$, then there is an integer $p \in \lcro 0,n\rcro$
and there are non-singular matrices $P,Q,R$ of $\Mat_n(\K)$ such that
$$V=P\,V_p\,Q \quad \text{and} \quad W=Q^{-1}\,W_p\,R$$
where, for $k \in \lcro 0,n\rcro$, we have set
$$V_k:=\biggl\{\begin{bmatrix}
0 & L \\
M & N
\end{bmatrix} \mid (L,M,N)\in \Mat_{1,n-k}(\K) \times \Mat_{n-1,k}(\K) \times \Mat_{n-1,n-k}(\K)\biggr\}$$
and
$$W_k:=\biggl\{\begin{bmatrix}
C & A \\
0 & B
\end{bmatrix} \mid (C,A,B)\in \Mat_{k,1}(\K) \times \Mat_{k,n-1}(\K) \times \Mat_{n-k,n-1}(\K)\biggr\}.$$
\end{enumerate}
\end{prop}

\begin{Rem}
A straightforward computation shows that, for every $p \in \lcro 0,n\rcro$,
one has $\codim V_p+\codim W_p=n$ whilst, for every pair
$(B,C) \in V_p \times W_p$, the product $BC$ has $0$ as entry at the $(1,1)$ spot,
hence $E_{1,1}$ is not a linear combination of matrices in $V_p \cdot W_p$. \\
In particular, this proves that the upper bound in point (a) is tight.
\end{Rem}

\begin{proof}
Assume that $\codim V+\codim W \leq n$. Set $\calA:=V \cdot W$.
We wish to prove that $(V\cdot W)^\bot=\{0\}$ save for a few special cases.
Let $D \in \calA^\bot$. Set $B \in V$. Then $\forall C \in W, \; \tr(DBC)=0$.
The linear map
$$f_D : \begin{cases}
\Mat_n(\K) & \longrightarrow \Mat_n(\K) \\
B & \longmapsto D\,B
\end{cases}$$
thus maps $V$ into $W^\bot$.
However, $f_D$ is represented in a well-chosen basis by the matrix $D \otimes I_n$, with rank $n\,\rk D$,
hence $\dim \Ker f_D=n\,(n-\rk D)$.
By the rank theorem, we deduce that
$$\dim V \leq \dim \Ker f_D+\dim W^\bot
=n\,(n-\rk D)+\codim W$$
hence
$$\codim V+\codim W \geq n\,\rk D.$$
If $\codim V+\codim W<n$, this shows $D=0$, hence $\calA^\bot=\{0\}$, and we deduce that $\Vect \calA=\Mat_n(\K)$. \\
Assume now that $\codim V+\codim W=n$ and $\calA^\bot \neq \{0\}$, and choose $D \in \calA^\bot \setminus \{0\}$. Then
$\rk D=1$. Notice then that $\codim V+\codim W \leq n\,\rk D$, so the rank theorem shows that
$f_D(V)=W^\bot$ and $\Ker f_D \subset V$. A similar line of reasoning shows that
$$g_D :\begin{cases}
\Mat_n(\K) & \longrightarrow \Mat_n(\K) \\
C & \longmapsto CD
\end{cases}$$
satisfies $\Ker g_D \subset W$.
Since $\rk D=1$, there are non-singular matrices $P$ and $R$ such that $D=PE_{1,1}R$. Replacing
$V$ and $W$ respectively with $RV$ and $WP$, we may assume $D=E_{1,1}$.
Then the inclusions $\Ker f_D \subset V$ and $\Ker g_D \subset W$ show that
$V$ contains every matrix of the form $\begin{bmatrix}
0 \\
M
\end{bmatrix}$ for some $M \in \Mat_{n-1,n}(\K)$, and every matrix of the form
$\begin{bmatrix}
0 & N \\
\end{bmatrix}$ for some $N \in \Mat_{n,n-1}(\K)$.
We may then find linear subspaces $E$ and $F$ respectively of $\Mat_{1,n}(\K)$ and $\Mat_{n,1}(\K)$ such that
$$V=\biggl\{\begin{bmatrix}
L \\
M
\end{bmatrix} \mid L \in E, \; M \in \Mat_{n-1,n}(\K)\biggr\} \quad \text{and} \quad
W=\biggl\{\begin{bmatrix}
C & N
\end{bmatrix} \mid C \in F, \; N \in \Mat_{n,n-1}(\K)\biggr\},$$
with $2n-\dim E-\dim F=\codim V+\codim W$, hence $\dim E+\dim F=n$. \\
The hypothesis $D \in \calA^\bot$ yields $LC=0$ for every $(L,C)\in E \times F$. \\
Setting $p:=n-\dim E$ and choosing a non-singular matrix $Q$ such that $E\,Q=\Bigl\{\begin{bmatrix}
0 & L_1
\end{bmatrix} \mid L_1 \in \Mat_{1,n-p}(\K)\Bigr\}$, we may replace $V$ with $V\,Q$ and $W$ with $Q^{-1}W$.
In this situation, we still have $E_{1,1} \in \calA^\bot$, and we now learn that
$$F \subset \biggl\{\begin{bmatrix}
C_1 \\
0
\end{bmatrix} \mid C_1 \in \Mat_{p,1}(\K)\biggr\}.$$
Since $\dim F=n-p$, we deduce that
this inclusion is an equality, which finally shows that $V=V_p$ and $W=W_p$.
\end{proof}

\section{The semigroup generated by a large affine subspace}\label{productofanysize}

\subsection{Starting the induction}\label{startinduc}

We will prove Theorem \ref{prodall} by establishing the slightly stronger statement:

\begin{theo}\label{prodallaffine}
Let $\calV$ be an affine subspace of $\Mat_n(\K)$ such that $\codim \calV<n-1$. \\
Then $\Mat_n(\K)=\calV^{(\infty)}$.
\end{theo}

\noindent Note that the result trivially holds when $n \leq 2$.
We will now proceed by induction. We fix an integer $n \geq 3$ and assume Proposition \ref{prodallaffine}
holds for every affine subspace of $\Mat_{n-1}(\K)$ with a codimension lesser than $n-2$.
In the rest of the proof, we fix an affine subspace $\calV$ of $\Mat_n(\K)$ such that
$\codim \calV<n-1$. We let $V$ denote its translation vector space.

\subsection{Reduction to the case of non-singular matrices}\label{singular}

In this section, we make the following assumption:
\begin{center}
Every matrix of $\GL_n(\K)$ is a product of matrices of $\calV$.
\end{center}
We will prove right away that this entails that every matrix of $\Mat_n(\K)$ is a product of matrices of $\calV$.
Classically, there are three steps:

\begin{enumerate}[(i)]
\item $\calV$ contains a rank $n-1$ matrix;
\item $\calV^{(\infty)}$ contains every rank $n-1$ matrix of $\Mat_n(\K)$;
\item $\calV^{(\infty)}$ contains every singular matrix of $\Mat_n(\K)$.
\end{enumerate}

\begin{proof}[Proof of step (i)]
The linear subspace $V^\bot$ has dimension lesser than $n$ hence
there is an integer $i \in \lcro 1,n\rcro$ such that $V^\bot$ contains no non-zero matrix
with all columns zero save for the $i$-th. Conjugating by a permutation matrix, we lose no generality by assuming
$V^\bot$ contains no non-zero matrix with all columns zero save for the $n$-th. This shows that $f : M \mapsto L_n(M)$
is a surjective affine map from $\calV$ to $\Mat_{1,n}(\K)$ (where $L_n(M)$ denotes the $n$-th row of $M$).
Then $\calW:=f^{-1} \{0\}$ is an affine subspace of $\calV$ with $\dim \calW=\dim \calV-n>n^2-(2n-1)$.
We write then every $M \in \calW$ as
$$M=\begin{bmatrix}
\alpha(M) \\
0
\end{bmatrix} \quad \text{with $\alpha(M) \in \Mat_{n-1,n}(\K)$.}$$
Then $\alpha(\calW)$ is an affine subspace of $\Mat_{n-1,n}(\K)$ and $\dim \alpha(\calW)>n(n-2)$.
Using our generalization of Dieudonn\'e's theorem for affine subspaces (cf.\ Theorem 6 of \cite{dSPaffpres}), we deduce that
$\alpha(\calW)$ contains a rank $n-1$ matrix, hence $\calV$ has a rank $n-1$ element.
\end{proof}

\begin{proof}[Proof of step (ii)]
Let $A \in \Mat_n(\K)$ be a rank $r$ matrix. If $\calV^{(\infty)}$ contains a rank $r$ matrix $B$,
then there are non-singular matrices $P$ and $Q$ such that $A=P\,B\,Q$, hence the preliminary assumption
shows that $A \in \calV^{(\infty)}$. Step (ii) follows then readily from step (i).
\end{proof}

\begin{proof}[Proof of step (iii)]
Let $r \in \lcro 0,n-1\rcro$.
Then the rank $r$ matrix $J_r:=\begin{bmatrix}
I_r & 0 \\
0 & 0
\end{bmatrix}$ decomposes as a product $J_r=\underset{k=r+1}{\overset{n}{\prod}}(I_n-E_{k,k})$
of rank $n-1$ matrices, hence it belongs to $\calV^{(\infty)}$ by step (ii).
The argument from step (ii) then shows that $\calV^{(\infty)}$ contains every rank $r$ matrix of $\Mat_n(\K)$.
\end{proof}

\noindent It now suffices to prove that $\GL_n(\K) \subset \calV^{(\infty)}$.

\subsection{A good situation}\label{goodcase}

Recall that $V$ denotes the translation vector space of $\calV$, and set
$$H:=V \cap \Vect(E_{1,2},\dots,E_{1,n}).$$
For every $N \in H$, we write
$$N=\begin{bmatrix}
0 & L(N) \\
0 & 0
\end{bmatrix} \quad \text{with $L(N) \in \Mat_{1,n-1}(\K)$.}$$
Then $L(H)$ is a linear subspace of $\Mat_{1,n-1}(\K)$ and the rank theorem shows that
$$\dim L(H)=\dim H \geq (n-1)-\codim_{\Mat_n(\K)} V>0.$$
Hence $L(H)$ contains a non-zero matrix (this will be of crucial interest later on).

\vskip 3mm
\noindent
Given $M \in \Mat_n(\K)$, we let $C_1(M)$ denote its first column.
We consider the affine map
$$(C_1)_{|\calV} : \begin{cases}
\calV & \longrightarrow \Mat_{n,1}(\K) \\
M & \longmapsto C_1(M).
\end{cases}$$
Let us make a first assumption:
\begin{itemize}
\item[(i)] $(C_1)_{|\calV}$ is onto.
\end{itemize}
Then
$$\calW:=\Bigl\{M \in \calV : \quad C_1(M)=\begin{bmatrix}
1 &  0 & \cdots & 0 \end{bmatrix}^T\Bigr\}$$ is an affine subspace of $\calV$ with $\dim \calW=\dim \calV-n$. \\
For every $M \in \calW$, we write
$$M=\begin{bmatrix}
1 & L(M) \\
0 & K(M)
\end{bmatrix} \quad \text{with $K(M) \in \Mat_{n-1}(\K)$ and $L(M) \in \Mat_{1,n-1}(\K)$.}$$
Finally, we consider the affine subspace $K(\calW)$ of $\Mat_{n-1}(\K)$.
Our second assumption will be:
\begin{itemize}
\item[(ii)] $\codim_{\Mat_{n-1}(\K)} K(\calW)<n-2$.
\end{itemize}

\noindent From there, we will show that every matrix of $\GL_n(\K)$ belongs to $\calV^{(\infty)}$.
Let $M \in \GL_n(\K)$. Then $C_1(M) \neq 0$. We first prove that $C_1(M)$ is also
the first column of a non-singular matrix of $\calV$:

\begin{lemme}
Let $\calV'$ be an affine subspace of $\Mat_n(\K)$ such that $\codim \calV'<n-1$.
Let $C \in \Mat_{n,1}(\K) \setminus \{0\}$ and assume some element of $\calV'$ has $C$ as first column.
Then some element of $\calV' \cap \GL_n(\K)$ has $C$ as first column.
\end{lemme}

\begin{proof}
Set $C_0:=\begin{bmatrix}
1 &  0 & \cdots & 0 \end{bmatrix}^T$. Choosing $P \in \GL_n(\K)$ such that $P\,C=C_0$ and replacing $\calV'$ with $P\,\calV'$,
we may assume $C=C_0$. With the above notations (though not assuming that $N \mapsto C_1(N)$ maps
$\calV'$ onto $\Mat_{n,1}(\K)$), we obtain that $\calW' \neq \emptyset$, hence the rank theorem shows
$\codim_{\Mat_{n-1}(\K)}K(\calW')<n-1$. Dieudonn\'e's theorem for affine subspaces \cite{Dieudonne} then shows
that the affine subspace $K(\calW')$ contains a non-singular matrix, QED.
\end{proof}

\noindent From there, we may choose some $N \in \calV \cap \GL_n(\K)$ with $C_1(M)$ as first column.
The matrix $A:=N^{-1}M$ is then non-singular and has the form
$$A=\begin{bmatrix}
1 & * \\
0 & P
\end{bmatrix} \quad \text{for some $P \in \GL_{n-1}(\K)$.}$$
It thus suffices to prove that $A \in \calV^{(\infty)}$. This will come from the next proposition:

\begin{prop}
Assuming conditions (i) and (ii) hold, let $P \in \GL_{n-1}(\K)$ and $L \in \Mat_{1,n-1}(\K)$. Then the matrix
$\begin{bmatrix}
1 & L \\
0 & P
\end{bmatrix}$ belongs to $\calW^{(\infty)}$.
\end{prop}

\begin{proof}
Condition (ii) and the induction hypothesis yield matrices $P_1,\dots,P_r$ in $K(\calW)$ such that
$P=P_1P_2\cdots P_r$, hence there are row matrices $L_1,\dots,L_r$ in $\Mat_{1,n-1}(\K)$ such that:
\begin{itemize}
\item $Q_k:=\begin{bmatrix}
1 & L_k \\
0 & P_k
\end{bmatrix}$ belongs to $\calV$ for every $k \in \lcro 1,r\rcro$;
\item $Q_1Q_2\cdots Q_r
=\begin{bmatrix}
1 & L' \\
0 & P
\end{bmatrix}$ for some $L' \in \Mat_{1,n-1}(\K)$.
\end{itemize}
In order to conclude, it suffices to prove that the matrix $\begin{bmatrix}
1 & L-L' \\
0 & I_{n-1}
\end{bmatrix}$ belongs to $\calV^{(\infty)}$, since left-multiplying it by $\begin{bmatrix}
1 & L' \\
0 & P
\end{bmatrix}$ yields $\begin{bmatrix}
1 & L \\
0 & P
\end{bmatrix}$. \\
We actually prove that $\calV^{(\infty)}$ contains
$\begin{bmatrix}
1 & L_1 \\
0 & I_{n-1}
\end{bmatrix}$ for every $L_1 \in \Mat_{1,n-1}(\K)$. Notice that the set $\calA$ of those
$L_1 \in \Mat_{1,n-1}(\K)$ such that $\begin{bmatrix}
1 & L_1 \\
0 & I_{n-1}
\end{bmatrix} \in \calV^{(\infty)}$ is closed under sum because $\calV^{(\infty)}$ is closed under product. \\
Let $R \in \GL_n(\K)$.
By the previous line of reasoning, there are matrices
$Q_1=\begin{bmatrix}
1 & L_1 \\
0 & P_1
\end{bmatrix},\dots,Q_r=\begin{bmatrix}
1 & L_r \\
0 & P_r
\end{bmatrix}$ in $\calW$ and a row matrix $L' \in \Mat_{1,n-1}(\K)$ such that $Q_1\cdots Q_r=\begin{bmatrix}
1 & L' \\
0 & R^{-1}
\end{bmatrix}$. Also, there is a row matrix $L'' \in \Mat_{1,n-1}(\K)$ such that $\begin{bmatrix}
1 & L'' \\
0 & R
\end{bmatrix}$ belongs to $\calW^{(\infty)}$. \\
Notice that $L_r$ may be replaced with $L_r+L_0$ for any $L_0 \in L(H)$ (recall the definition of $L(H)$ from the beginning of the section): it follows that
$\begin{bmatrix}
1 & L'+L_0 \\
0 & R^{-1}
\end{bmatrix} \in \calV^{(\infty)}$ for any $L_0 \in L(H)$.
Right-multiplying this matrix by $\begin{bmatrix}
1 & L'' \\
0 & R
\end{bmatrix}$, we deduce that $L'R+L''+L_0R$ belongs to $\calA$ for every $L_0 \in L(H)$.
We have thus found, for every $R \in \GL_n(\K)$, a row matrix $L_R \in \Mat_{1,n-1}(\K)$
such that $L_R+L(H)\,R \subset \calA$. \\
Recall from the beginning of this paragraph that there is a non-zero $E \in L(H)$. \\
We may then find non-singular matrices $P_1,\dots,P_{n-1}$ such that $(EP_i)_{1 \leq i \leq n-1}$
is a basis of $\Mat_{1,n-1}(\K)$. Since $\calA$ is closed under addition and $L(H)$ is a linear subspace of
$\Mat_{1,n-1}(\K)$, we deduce that $\calA$ contains $\sum_{k=1}^{n-1}L_{P_k}+\Vect(EP_k)_{1 \leq k \leq n-1}$,
which clearly equals $\Mat_{1,n-1}(\K)$. Hence $\calA=\Mat_{1,n-1}(\K)$, QED.
\end{proof}

\subsection{Why the good situation almost always arises up to conjugation}

Notice first that given $P \in \GL_n(\K)$,
one has $(P\calV P^{-1})^{(\infty)}=P\,\calV^{(\infty)}\,P^{-1}$, so
we may replace $\calV$ with any conjugate affine subspace
in order to prove that $\calV^{(\infty)}=\Mat_n(\K)$.
We denote by $(e_1,\dots,e_n)$ the canonical basis of $\K^n$.

\vskip 2mm
\noindent
Here, we prove the following result:
\begin{prop}\label{alwaysgood}
Let $\calV$ be an affine subspace of $\Mat_n(\K)$ such that $\codim \calV<n-1$.
Then :
\begin{enumerate}[(a)]
\item Either $n=3$ and there exists $a \in \K$ such that $\calV=\bigl\{M \in \Mat_3(\K): \; \tr M=a\bigr\}$;
\item Or there exists $P \in \GL_n(\K)$ such that $P\,\calV\,P^{-1}$ satisfies conditions
(i) and (ii) of Section \ref{goodcase}.
\end{enumerate}
\end{prop}

\noindent Before proving this, we must analyze condition (i)
in terms of the structure of $V^\bot$, where $V$ denotes the translation vector space of $\calV$.
For $M \mapsto C_1(M)$ not to be onto from $\calV$, it is necessary and sufficient for it
not to be onto from $V$, which is equivalent to the existence of a non-zero row matrix
$L \in \Mat_{1,n}(\K)$ such that $\begin{bmatrix}
L \\
0
\end{bmatrix} \in V^\bot$.
Hence (i) holds if and only if no matrix $A$ in $V^\bot$
satisfies $\im A=\Vect(e_1)$.

\noindent Assume now that condition (i) holds. The rank theorem shows:
$$\codim_{\Mat_{n-1}(\K)} K(\calW) \leq \codim_{\Mat_n(\K)} \calV< n-1.$$
If (ii) does not hold, then the rank theorem shows that
$\codim_{\Mat_n(\K)} \calV=n-2$ and $\dim L(H)=n-1$, hence $L(H)=\Mat_{1,n-1}(\K)$:
it would follow that $V$ contains every
matrix $A \in \frak{sl}_n(\K)$ such that $\im A=\Vect(e_1)$.

\noindent We deduce that conditions (i) and (ii) hold in the case
$V^\bot$ contains no rank $1$ matrix with image $\Vect(e_1)$ \emph{and}
$V$ does not contain every matrix $A \in \frak{sl}_n(\K)$ with image $\Vect(e_1)$.
With that in mind, we may now prove Proposition \ref{alwaysgood}.

\begin{proof}[Proof of Proposition \ref{alwaysgood}]
We reason in terms of linear operators.
We use the canonical basis to identify $\calV$ with an affine space of linear endomorphisms of $\K^n$.
The symmetric bilinear form $(A,B) \mapsto \tr(AB)$ on $\Mat_n(\K)$ then corresponds to
$(u,v) \mapsto \tr(u \circ v)$. \\
We assume there is no $P \in \GL_n(\K)$ such that
$P\,\calV\,P^{-1}$ satisfies conditions (i) and (ii) of Section \ref{goodcase}.
By the above remarks, this shows that for every 1-dimensional linear subspace $D \subset \K^n$
for which $V^\bot$ contains no endomorphism with image $D$,
one has $u \in V$ for every $u \in \frak{sl}(\K^n)$ such that $\im u=D$. \\
We then wish to show that $V$ contains every trace $0$ endomorphism.
\begin{itemize}
\item Consider the linear subspace $U$ of $V^\bot$ spanned by its rank $1$ endomorphisms.
In $U$, we choose a basis $(u_1,\dots,u_r)$ consisting of rank $1$ endomorphisms,
and we set $F:=\im u_1+\cdots+\im u_r \subset \K^n$.
Then every rank $1$ element in $V^\bot$ has its image included in $F$ and
$$\dim F \leq r \leq \dim V^\bot \leq n-2.$$
\item It follows that $V$ contains every $u \in \frak{sl}(\K^n)$ such that
$\rk u=1$ and $\im u \not\subset F$. We will let $\calB$ denote the set of those endomorphisms.
\item Notice that the set of rank $1$ endomorphisms of $\K^n$ with trace $0$
spans $\bigl\{u \in \calL(\K^n) :\; \tr u=0\}$: it suffices to consider the
matrices $E_{i,j}$ and $E_{j,i}$, for $1 \leq i<j \leq n$, and the matrices
$E_{1,1}+E_{k,1}-E_{1,k}-E_{k,k}$, for $2 \leq k \leq n$.
\item We finish by proving that every $u \in \calL(\K^n)$ with rank $1$ and trace $0$
is a linear combination of elements of $\calB$. Set $u \in \calL(\K^n)$ such that
$\rk u=1$, $\tr u=0$ and $\im u \subset F$.
Choose $x_1 \in \im u \setminus \{0\}$.
Since $\codim F \geq 2$, we may choose $x_2 \in E \setminus (F \cup \Ker u)$ and
then $x_3 \in E$ such that $\Vect(x_2,x_3) \cap F=\{0\}$.
We finally extend $(x_1,x_2,x_3)$ into a basis $\bfB$ of $\K^n$ using vectors of $\Ker u$. \\
Then there is a matrix $A \in \Mat_3(\K)$, of the form $A=\begin{bmatrix}
0 & L \\
0 & 0
\end{bmatrix}$ for some $L \in \Mat_{1,2}(\K) \setminus \{0\}$, such that
$$M_\bfB(u)=\begin{bmatrix}
A & 0 \\
0 & 0
\end{bmatrix}.$$
Since $\Vect(x_1,x_2,x_3) \cap F=\Vect(x_1)$, we deduce: for every
$A_1 \in \frak{sl}_3(\K)$ such that $\rk A_1=1$ and $\im A_1 \neq \Vect \begin{bmatrix}
1 & 0 & 0\end{bmatrix}^T$,
there is some $v \in \calB$ such that
$M_\bfB(v)=\begin{bmatrix}
A_1 & 0 \\
0 & 0
\end{bmatrix}$. In order to conclude, it thus suffices to solve the case $n=3$. \\
By a change of basis, it suffices to prove that the vector space
$\frak{sl}_3(\K)$ is spanned by its rank $1$ matrices
whose image is different from $\Vect \begin{bmatrix}1 & 1 & 1\end{bmatrix}^T$. This is obvious using the family from the preceding
bullet-point.
\end{itemize}
Finally, we have shown that $\frak{sl}_n(\K) \subset V$. If
$V=\Mat_n(\K)$, then conditions (i) and (ii) of Section \ref{goodcase} obviously hold.
If not, one has $\frak{sl}_n(\K)=V$ thus $\calV=\bigl\{M \in \Mat_n(\K): \; \tr M=a\bigr\}$
for some $a \in \K$. Then condition (i) is clearly satisfied by $\calV$,
and since (ii) is not, one has $\codim_{\Mat_n(\K)}\calV=n-2$
(see the remarks above the present proof). Since $\calV$ is a hyperplane of $\Mat_n(\K)$, we finally deduce that $n=3$.
\end{proof}

\subsection{The exceptional case}

Combining Proposition \ref{alwaysgood}
with the arguments from Sections \ref{singular} and \ref{goodcase}, it is clear that
our proof of Theorem \ref{prodallaffine} will be complete when the following result will be established:

\begin{prop}
Let $a \in \K$ and set $\calH:=\bigl\{M \in \Mat_3(\K) : \; \tr M=a\bigr\}$.
Then $\GL_3(\K) \subset \calH^{(\infty)}$.
\end{prop}

\begin{proof}
Notice that $\calH$ is closed under conjugation hence $\calH^{(\infty)}$ also is.
\begin{itemize}
\item Assume first that $\card \K>2$. Then the union of the conjugacy classes of
$\Diag(\lambda,1,1)$ for $\lambda \in \K \setminus \{0,1\}$ generates\footnote{
By \cite{Lang} Proposition 9.1 p.541, it suffices to prove that some transvection matrix
is a product of matrices of the aforementioned set. Choosing $\lambda \in \K \setminus \{0,1\}$, we see that
$\begin{bmatrix}
1 & 1 & 0 \\
0 & 1 & 0 \\
0 & 0 & 1
\end{bmatrix}=\begin{bmatrix}
\lambda & 1-\lambda & 0 \\
0 & 1 & 0 \\
0 & 0 & 1
\end{bmatrix}\times \begin{bmatrix}
\lambda^{-1} & 1 & 0 \\
0 & 1 & 0 \\
0 & 0 & 1
\end{bmatrix}$ with
$\begin{bmatrix}
\lambda^{-1} & 1 & 0 \\
0 & 1 & 0 \\
0 & 0 & 1
\end{bmatrix} \sim \Diag(\lambda^{-1},1,1)$
and
$\begin{bmatrix}
\lambda & 1-\lambda & 0 \\
0 & 1 & 0 \\
0 & 0 & 1
\end{bmatrix} \sim \Diag(\lambda,1,1)$.}
the group $\GL_3(\K)$.
Notice that this subset is closed under inversion hence every matrix of  $\GL_3(\K)$ is  a product of matrices in this
subset. \\
For every $\lambda \in \K \setminus \{0,1\}$, remark that
$$\begin{bmatrix}
a-1 & 1 & 0 \\
1 & 0 & 0 \\
0 & 0 & 1
\end{bmatrix}\times
\begin{bmatrix}
0 & \lambda & 0 \\
1 & a-1 & 0 \\
0 & 0 & 1
\end{bmatrix}=
\begin{bmatrix}
1 & (\lambda+1)\,(a-1) & 0 \\
0 & \lambda & 0 \\
0 & 0 & 1
\end{bmatrix} \sim \Diag(\lambda,1,1),$$
hence $\Diag(\lambda,1,1)$ belongs to $\calH^{(\infty)}$. This shows $\GL_3(\K) \subset \calH^{(\infty)}$.
\item Assume now $\card \K=2$. Then every matrix of $\GL_3(\K)=\SL_3(\K)$ is a product of matrices all similar to the transvection matrix
$T:=\begin{bmatrix}
1 & 1 & 0 \\
0 & 1 & 0 \\
0 & 0 & 1
\end{bmatrix}$ (see \cite{Lang} Proposition 9.1 p.541). If $a=1$, we then see that
$$T=\begin{bmatrix}
1 & 1 & 0 \\
0 & 1 & 0 \\
0 & 0 & 1
\end{bmatrix} \times \begin{bmatrix}
1 & 0 & 0 \\
0 & 1 & 0 \\
0 & 0 & 1
\end{bmatrix} \in \calH^{(2)}.$$
If $a=0$, we write:
$$T=\begin{bmatrix}
0 & 1 & 1 \\
0 & 0 & 1 \\
1 & 0 & 0
\end{bmatrix}\times \begin{bmatrix}
0 & 0 & 1 \\
1 & 0 & 0 \\
0 & 1 & 0
\end{bmatrix} \in \calH^{(2)}.$$
In any case, we deduce that $\GL_3(\K) \subset \calH^{(\infty)}$.
\end{itemize}
\end{proof}

\noindent This completes the proof of Theorem \ref{prodallaffine} by induction.

\section{Products of two matrices from an hyperplane}\label{hyperplanes}

In this section, we consider a (linear) hyperplane $H$ of $\Mat_n(\K)$. If $n \geq 3$, then
Theorem \ref{prodall} shows that every matrix of $\Mat_n(\K)$ is a product of matrices
from $H$ (possibly with a large number of factors). Here, we will
see that actually two matrices always suffice in the product.
As a warm up, we start by considering the case $n=2$
and by classifying all the counter-examples.

\noindent The following basic lemma of affine geometry will be of constant use:

\begin{lemme}\label{affinelemma}
Let $F$ be a linear hyperplane of a vector space $E$, and
$\calG$ be an affine subspace of $E$ with translation vector space $G$. If $F \cap \calG=\emptyset$, then $G \subset F$.
\end{lemme}

\begin{proof}
Assume $G \not\subset F$. Then $F+G=E$ since $F$ is a linear hyperplane of $E$.
Choosing $a \in \calG$ and writing it $a=x+y$ for some $(x,y)\in F \times G$, we then see that
$a-y \in F \cap \calG$, hence $F \cap \calG \neq \emptyset$.
\end{proof}

\subsection{The case $n=2$}

\noindent Here, we prove the following result:

\begin{prop}\label{hypern=2}
Let $H$ be a linear hyperplane of $\Mat_2(\K)$.
Then every matrix of $\Mat_2(\K)$ is a product of two elements of $H$
unless $H$ is conjugate to one of the following hyperplanes
$$H_0:=\Biggl\{\begin{bmatrix}
0 & b \\
a & c
\end{bmatrix} \; \mid \; (a,b,c)\in \K^3\Biggr\}
 \quad \text{and} \quad T_2^+(\K):=\Biggl\{\begin{bmatrix}
a & b \\
0 & c
\end{bmatrix} \; \mid \; (a,b,c)\in \K^3\Biggr\}.
$$
\end{prop}

\begin{Rem}
Since $T_2^+(\K)$ is a strict subalgebra of $\Mat_2(\K)$, it clearly does
not verify the result under scrutiny, and neither does any of its conjugate hyperplanes. \\
On the other hand, the matrix
$A=\begin{bmatrix}
0 & 1 \\
1 & 0
\end{bmatrix}$ cannot be decomposed as $A=BC$ for some pair $(B,C)\in H_0^2$.
If indeed it could, then $C$ would be non-singular, hence
$C^{-1}=\begin{bmatrix}
a & b \\
c & 0
\end{bmatrix}$ for some triple $(a,b,c)\in \K^3$ with $b \neq 0$ and $c \neq 0$,
and equating $B$ with $A\,C^{-1}$ would yield a contradiction (this would mean $B$ has $c \neq 0$
as entry at the $(1,1)$ spot).
\end{Rem}

\begin{proof}[Proof of Proposition \ref{hypern=2}]
We assume $H$ is neither conjugate to $H_0$ nor to $T_2(\K)^+$.
Choose an non-zero matrix $A$ in the line $H^\bot$. Then $A$
is conjugate to neither $\begin{bmatrix}
0 & 1 \\
0 & 0
\end{bmatrix}$ nor to $\begin{bmatrix}
\lambda & 0 \\
0 & 0
\end{bmatrix}$ for some $\lambda \neq 0$. This shows $A$ is non-singular
(if not, then $A$ has rank $1$ hence is conjugate to one of the aforementioned matrices).
We let $M \in \Mat_2(\K) \setminus \{0\}$ and try to decompose $M$ as a product of two matrices in $H$.
\begin{itemize}
\item \emph{The case $M$ is non-singular.} \\
For $N \in \Mat_2(\K)$, we let $\Com(N)$ denote its matrix of cofactors.
The map $N \mapsto \Com(N)$ is a linear automorphism of $\Mat_2(\K)$, hence
$$V:=\Bigl\{M \Com(N)^T\mid N \in H\Bigr\}$$
is a hyperplane of $\Mat_2(\K)$.
If $V \cap H$ contains a non-singular matrix $B$, then we have a matrix $C \in H$
such that $M\,\Com(C)^T=B$, hence $C$ is non-singular and
$M=B\,\bigl(\frac{1}{\det(C)}\cdot C\bigr)$ belongs to $H^{(2)}$. \\
Assume now that all the matrices in $V \cap H$ are singular.
Since $\dim (V \cap H) \geq 2$, we deduce that $H$ contains a two-dimensional singular linear subspace
(i.e.\ one that contains no non-singular matrix). Replacing $H$
with a conjugate hyperplane, we may use Lemma 32.1 of \cite{Prasolov}
and assume, without loss of generality, that $H$ contains one of the planes
$$\Biggl\{\begin{bmatrix}
a & 0 \\
b & 0
\end{bmatrix} \; \mid \; (a,b)\in \K^2\Biggr\}
 \quad \text{or} \quad
\Biggl\{\begin{bmatrix}
a & b \\
0 & 0
\end{bmatrix} \; \mid \; (a,b)\in \K^2\Biggr\}.$$
However, in the first case, the first row of $A$ is zero, and in the second case, the first column of $A$ is zero, contradicting
the non-singularity of $A$. This completes the case $M$ is non-singular.

\item \emph{The case $M$ is singular.} \\
Then $\rk M=1$ and we may choose a non-zero vector $e_1 \in \Ker M$ and extend it into a basis
$(e_1,e_2)$ of $\K^2$.
Since $\{N \in \Mat_2(\K) :\; e_1 \in \Ker N\}$ is a linear plane,
it has a common non-zero matrix $C$ with $H$. \\
We now search for some $B \in H$ satisfying $M=B\,C$. \\
First of all, since $\rk C=\rk M$ and $e_1 \in \Ker C$, there is some $B_0 \in \Mat_2(\K)$ such that
$M=B_0\,C$.
Then $\calP:=\bigl\{B \in \Mat_2(\K) : \; B\,C=M\bigr\}$
is a plane with translation vector space $P:=\bigl\{B \in \Mat_2(\K) : \; B\,C=0\}$. \\
If $\calP \cap H \neq \emptyset$, then we find some $B \in H$ such that $M=B\,C$.
If not, Lemma \ref{affinelemma} would show that $P \subset H$,
which would yield the same contradiction as in the case $M$ is non-singular (we would find that $A$
is singular). This completes the case $M$ is singular.
\end{itemize}
\end{proof}

\subsection{The case $n \geq 3$}

Here, we assume $n \geq 3$, we let $H$ be a linear hyperplane of $\Mat_n(\K)$, and we choose a non-zero matrix $A$
in $H^\bot$. Letting $M \in \Mat_n(\K) \setminus \{0\}$, we try to decompose $M$ as the product of two matrices in $H$.

\subsubsection{The case $M$ is singular}

Up to conjugation by a well-chosen non-singular matrix, we may assume the first row of $A$ is non-zero.
We denote by $(e_1,\dots,e_n)$ the canonical basis of $\K^n$.
The basic idea is to find a matrix $C$ in $H$ with the same kernel as $M$, and then
another $B \in H$ such that $A=B\,C$ (notice the similarity with the case $n=2$).
Set $p:=\rk M$, so that $1 \leq p<n$.

\vskip 2mm
\begin{itemize}
\item The set
$$V:=\bigl\{C \in \Mat_n(\K) : \, \Ker M \subset \Ker C \; \text{and} \; \im C \subset
\Vect(e_2,\dots,e_n)\bigr\}$$ is a linear subspace of $\Mat_n(\K)$ with dimension $(n-1)\,p$,
and $\forall C \in V, \; \rk C \leq p$.
\item It follows that $V \cap H$ has a dimension greater than or equal to $(n-1)p-1$
and $\forall C \in V\cap H, \; \rk C \leq p$.
Notice that $V \cap H$ is naturally isomorphic to a linear subspace of
$\Mat_{n-1,p}(\K)$ (through a rank-preserving map).
If $V \cap H$ contained no rank $p$ matrix, the Flanders-Meshulam theorem \cite{Meshulam} would show that
$\dim (V \cap H) \leq (n-1)(p-1)$.
However, since $n>2$, one has $(n-1)(p-1)<np-p-1$, hence $V \cap H$ contains a rank $p$ matrix $C$.
Therefore, $\rk M=\rk C$ and $\Ker M \subset \Ker C$, thus $\Ker M=\Ker C$
and it follows that $M=B_0\,C$ for some $B_0 \in \Mat_n(\K)$.
\item Define then the affine subspace $\calP:=\bigl\{B \in \Mat_n(\K) : \; B\,C=M\bigr\}$
with translation vector space $P:=\bigl\{B \in \Mat_n(\K) : \; B\,C=0\bigr\}$.
By a \emph{reductio ad absurdum}, let us assume that $\calP \cap H = \emptyset$.
Then Lemma \ref{affinelemma} shows that $P \subset H$.
However, since $\im C \subset \Vect(e_2,\dots,e_n)$,
it would follow that for any $C_1 \in \Mat_{n,1}(\K)$, the matrix
$\begin{bmatrix}
C_1 & 0 & \cdots & 0
\end{bmatrix}$ would belong to $H$. This would entail that the first row of $A$ is zero,
in contradiction with our first assumption. We conclude that $\calP \cap H \neq \emptyset$, which provides some
$B \in H$ such that $M=B\,C$.
\end{itemize}
This shows that $M \in \calH^{(2)}$ whenever $M$ is singular.

\subsubsection{The case $M$ is non-singular}

We will actually prove a somewhat stronger statement:

\begin{prop}\label{inverse}
Let $H_1$ and $H_2$ be two linear hyperplanes of $\Mat_n(\K)$, with $n \geq 3$.
Then there is a non-singular matrix $P \in H_1$ such that $P^{-1} \in H_2$.
\end{prop}

\noindent Before proving this, we readily show how this solves our problem.
Since $M$ is non-singular, $M^{-1}\,H$ is a linear hyperplane of $\Mat_n(\K)$.
Applying Proposition \ref{inverse} to the hyperplanes $H$ and $M^{-1}H$ yields a
non-singular matrix $P \in H$ such that $P^{-1}\in M^{-1} H$. Therefore
$P^{-1}=M^{-1}C$ for some $C \in H$, which shows $M=C\,P \in H^{(2)}$.

\begin{proof}[Proof of Proposition \ref{inverse}]
We will use a \emph{reductio ad absurdum} by assuming that no non-singular matrix $P \in H_1$ satisfies $P^{-1} \in H_2$. \\
Choose $A_1$ and $A_2$ respectively in $H_1^\bot\setminus \{0\}$ and $H_2^\bot\setminus \{0\}$.
We will use the block decompositions:
$$A_1=\begin{bmatrix}
\alpha & L_1 \\
C_1 & M_1
\end{bmatrix} \quad \text{and} \quad A_2=\begin{bmatrix}
\beta & L_2 \\
C_2 & M_2
\end{bmatrix}$$
where $(\alpha,\beta)\in \K^2$, $(L_1,L_2)\in \Mat_{1,n-1}(\K)^2$, $(C_1,C_2)\in \Mat_{n-1,1}(\K)^2$ and
$(M_1,M_2)\in \Mat_{n-1}(\K)^2$.

\vskip 2mm
To start with :
\begin{center}
We assume $C_1 \neq 0$.
\end{center}
We will then prove that $C_2=0$ and $M_2=0$. \\
Let $Q \in \GL_{n-1}(\K)$. For $X \in \Mat_{1,n-1}(\K)$, set
$$f(X):=\begin{bmatrix}
1 & X \\
0 & Q
\end{bmatrix} \in \GL_n(\K),$$
the inverse of which is
$$f(X)^{-1}=\begin{bmatrix}
1 & -XQ^{-1} \\
0 & Q^{-1}
\end{bmatrix}.$$
Since $C_1 \neq 0$ and $n \geq 3$, there exists $X_0 \in \Mat_{1,n-1}(\K) \setminus \{0\}$ such that
$f(X_0) \in H_1$.
Set then $F:=\bigl\{X \in \Mat_{1,n-1}(\K) : \;  X C_1=0\}$, so that
$f(X_0+X)\in H_1$ for every $X \in F$.
Then $\calG:=\bigl\{f(X_0+X)^{-1} \mid X \in F\bigr\}$ is an affine subspace of $\Mat_n(\K)$
with translation vector space
$$\Biggl\{\begin{bmatrix}
0 & -XQ^{-1} \\
0 & 0
\end{bmatrix} \mid X \in F\Biggr\}.$$
By our initial assumption, one must have $\calG \cap H_2 =\emptyset$,
hence Lemma \ref{affinelemma} shows that the translation vector space of $\calG$ is included in $H_2$, which proves
$$\forall X \in \Mat_{1,n-1}(\K), \; XC_1=0 \Rightarrow XQ^{-1}C_2=0.$$
Since this holds for every non-singular $Q$, since $\GL_{n-1}(\K)$ acts transitively on
$\Mat_{n-1,1}(\K) \setminus \{0\}$, and $F \neq \{0\}$ (because $C_1\neq 0$ and $n \geq 3$), we deduce that
$$C_2=0.$$
We now assume $M_2 \neq 0$ and prove that it leads to a contradiction.
The matrix $Q$ may now be chosen such that $f(0)^{-1} \in H_2$.
Indeed, by Dieudonn\'e's theorem for affine subspaces \cite{Dieudonne}, the hyperplane of $\Mat_{n-1}(\K)$ defined by the equation $\tr(M_2\,N)=-\beta$ contains a non-singular matrix, and it suffices to choose $Q$ as its inverse.
Since $C_2=0$, we now have $f(X_0)^{-1} \in H_2$ which is a contradiction
because $f(X_0)\in H_1$. We have thus proven:
$$M_2=0.$$
Let us sum up:
\begin{center}
If $e_1$ is not an eigenvector of $A_1$, then $\im A_2 \subset \Vect(e_1)$.
\end{center}
Since the assumptions are unaltered by simultaneously conjugating $H_1$ and $H_2$ by an arbitrary non-singular matrix, we deduce:
\begin{center}
For every non-zero vector $x \in \K^n$ which is not an eigenvector of $A_1$, one has
$\im A_2 \subset \Vect(x)$.
\end{center}
However $A_2 \neq 0$. It follows that, given two linearly independent vectors of $\K^n$,
one must be an eigenvector of $A_1$. Obviously, this shows that $A_1$ is diagonalisable.
Assume now that $A_1$ is not a scalar multiple of $I_n$.
\begin{itemize}
\item If $\card \K \geq 3$, then we may choose eigenvectors
$x$ and $y$ of $A_1$ associated to distinct eigenvalues, choose $\lambda \in \K \setminus \{0,1\}$,
and notice that the vectors $x+y$ and $x+\lambda.y$ are linearly independent although none is an eigenvector of $A_1$.
\item Assume now $\card \K=2$ and choose a linearly independent triple $(x,y,z)$ and a pair $(\lambda,\mu)\in \K^2$ of distinct scalars such that $x,y,z$ are eigenvectors of $A_1$ respectively associated to the eigenvalues $\lambda,\lambda,\mu$:
then $x+z$ and $y+z$ are linearly independent and none is an eigenvector of $A_1$.
\end{itemize}
We deduce that $A_1$ is a scalar multiple of $I_n$. Since the pair $(A_2,A_1)$ satisfies the same assumptions as $(A_1,A_2)$,
we also find that $A_2$ is a scalar multiple of $I_n$, hence $H_1=H_2=\frak{sl}_n(\K)$.
Finally, the permutation matrix $P:=E_{1,n}+\sum_{j=1}^{n-1} E_{j+1,j}$ belongs to
$\frak{sl}_n(\K)$, and so does its inverse $P^T$. This is the final contradiction, which proves our claim.
\end{proof}
\noindent This completes our proof of Theorem \ref{prod2}.

\noindent The reader will check that the preceding arguments may be generalized effortlessly
so as to yield:

\begin{theo}
Let $n \geq 3$ be an integer, and $H_1$ and $H_2$ be two linear hyperplanes of $\Mat_n(\K)$.
Then every $A \in \Mat_n(\K)$ splits as $A=B\,C$ for some $(B,C)\in H_1 \times H_2$.
\end{theo}

\end{document}